\documentclass{article}

\usepackage{amssymb, amsmath, amsthm, tikz, graphicx,lmodern,indentfirst,enumerate,verbatim}

\usepackage[titletoc,page]{appendix}
\usepackage[margin=1.4in]{geometry}

\newtheorem{thm}{Theorem}[section]
\newtheorem{lem}[thm]{Lemma}
\newtheorem{prop}[thm]{Proposition}

\theoremstyle{definition}
\newtheorem{defn}[thm]{Definition}

\theoremstyle{remark}
\newtheorem{rem}[thm]{Remark}

\theoremstyle{definition}
\newtheorem{question}[thm]{Question}

\numberwithin{equation}{section}

\newcommand*\circleD[1]{\tikz[baseline=(char.base)]{
            \node[shape=circle,draw,inner sep=0pt,minimum size=7mm] (char) {#1};}}

\newenvironment{tight_enumerate}{
\begin{enumerate}
  \setlength{\itemsep}{0pt}
  \setlength{\parskip}{0pt}
}{\end{enumerate}}

\newcommand{\rootsG}[8]{\circleD{#1}
            \begin{tabular}{cccccc}
             &&&#2&& \\
             #3&#4&#5&#6&#7&#8
             \end{tabular}}
\newcommand{\rootsE}[7]{\circleD{#1}
            \begin{tabular}{ccccc}
             &&#2&& \\
             #3&#4&#5&#6&#7
             \end{tabular}}
\newcommand{\rootsC}[9]{\circleD{#1}
            \begin{tabular}{ccccccc}
             &&&&#2&& \\
             #3&#4&#5&#6&#7&#8&#9
             \end{tabular}}

\makeatletter

\newcommand{\Rmnum}[1]{\expandafter\@slowromancap\romannumeral #1@}
\makeatother

\begin{document}

\title{Non-separability and complete reducibility: $E_n$ examples with an application to a question of K\"ulshammer}
\author{Tomohiro Uchiyama\\
National Center for Theoretical Sciences, Mathematics Division\\
No.~1, Sec.~4, Roosevelt Rd., National Taiwan University, Taipei, Taiwan\\
\texttt{email:t.uchiyama2170@gmail.com}}
\date{}
\maketitle

\begin{abstract}
Let $G$ be a simple algebraic group of type $E_n (n=6,7,8)$ defined over an algebraically closed field $k$ of characteristic $2$. We present examples of triples of closed reductive groups $H<M<G$ such that $H$ is $G$-completely reducible, but not $M$-completely reducible. As an application, we consider a question of K\"ulshammer on representations of finite groups in reductive groups. We also consider a rationality problem for $G$-complete reducibility and a problem concerning conjugacy classes. 
\end{abstract}

\section{Introduction}
Let $G$ be a connected reductive algebraic group defined over an algebraically closed field $k$ of characteristic $p$. In ~\cite[Sec.~3]{Serre-building}, J.P. Serre defined the following:
\begin{defn}
A closed subgroup $H$ of $G$ is \emph{$G$-completely reducible} ($G$-cr for short) if whenever $H$ is contained in a parabolic subgroup $P$ of $G$, $H$ is contained in a Levi subgroup $L$ of $P$. 
\end{defn}
This is a faithful generalization of the notion of semisimplicity in representation theory: if $G=GL_n(k)$, a subgroup $H$ of $G$ is $G$-cr if and only if $H$ acts semisimply on $k^n$~\cite[Ex.~3.2.2(a)]{Serre-building}. If $p=0$, the notion of $G$-complete reducibility agrees with the notion of reductivity~\cite[Props.~4.1, 4.2]{Serre-building}. In this paper, we assume $p>0$. In that case, if a subgroup $H$ is $G$-cr, then $H$ is reductive~\cite[Prop.~4.1]{Serre-building}, but the other direction fails: take $H$ to be a unipotent subgroup of order $p$ of $G=SL_2$. See~\cite{Stewart-nonGcr} for examples of connected non-$G$-cr subgroups. In this paper, by a subgroup of $G$, we always mean a closed subgroup.  

 Completely reducible subgroups have been much studied as important ingredients to understand the subgroup structure of connected reductive algebraic groups~\cite{Liebeck-Seitz-memoir}, \cite{Liebeck-Testerman-irreducible-QJM}, \cite{Thomas-irreducible-JOA}. Recently, studies of complete reducibility via Geometric Invariant Theory (GIT for short) have been fruitful~\cite{Bate-separability-TransAMS}, \cite{Bate-uniform-TransAMS}, \cite{Bate-geometric-Inventione}. In this paper, we use a recent result from GIT (Proposition~\ref{GIT}). 

Here is the first problem we consider in this paper. Let $H<M<G$ be a triple of reductive algebraic groups.  It is known to be hard to find such a triple with $H$ $G$-cr but not $M$-cr~\cite{Bate-separability-TransAMS},~\cite{Uchiyama-Separability-JAlgebra}. The only known such examples are~\cite[Sec.~7]{Bate-separability-TransAMS} for $p=2, G=G_2$ and~\cite{Uchiyama-Separability-JAlgebra} for $p=2, G=E_7$. Recall that a pair of reductive groups $G$ and $M$ is called a \emph{reductive pair}
if ${\rm Lie}\,M$ is an $M$-module direct summand of $\mathfrak{g}$. For more on reductive pairs, see~\cite{Goodbourn-reductivepairs}. Our main result is:

\begin{thm}\label{main}
Let $G$ be a simple algebraic group of type $E_6$ (respectively $E_7$, $E_8$) of any isogeny type defined over an algebraically closed field $k$ of characteristic $2$. Then there exist reductive subgroups $H<M$ of $G$ such that $H$ is finite, $M$ is semisimple of type $A_5 A_1$ (respectively $A_7$, $D_8$), $(G,M)$ is a reductive pair, and $H$ is $G$-cr but not $M$-cr. 
\end{thm}

In this paper, we present new examples with the properties of Theorem~\ref{main} giving an explicit description of the mechanism for generating such examples. We give $11$ examples for $G=E_6$, $1$ new example for $G=E_7$, and $2$ examples for $G=E_8$. We use \textsf{Magma}~\cite{Magma} for our computations. Recall that $G$-complete reducibility is invariant under isogenies~\cite[Lem.~2.12]{Bate-geometric-Inventione}; in Sections~3,4, and 5, we do computations for simply-connected $G$ only, but that is sufficient to prove Theorem~\ref{main} for $G$ of any isogeny type.

We recall a few relevant definitions and results from~\cite{Bate-separability-TransAMS}, \cite{Uchiyama-Separability-JAlgebra}, which motivated our work. We denote the Lie algebra of $G$ by $\rm{Lie}\,G=\mathfrak{g}$. 

\begin{defn}
Let $H$ and $N$ be subgroups of $G$ where $H$ acts on $N$ by group automorphisms. The action of $H$ is called \emph{separable} in $N$ if the global centralizer of $H$ in $N$ agrees with the infinitesimal centralizer of $H$ in $\textup{Lie}\;N$, that is, $C_N(H)= \mathfrak{c}_{{\rm Lie}\,N}(H)$. Note that the condition means that the set of fixed points of $H$ acting on $N$, taken with its natural scheme structure, is smooth. 
\end{defn}
This is a slight generalization of the notion of separable subgroups. Recall that

\begin{defn}
Let $H$ be a subgroup of $G$ acting on $G$ by inner automorphisms. Let $H$ act on $\mathfrak{g}$ by the corresponding adjoint action. Then $H$ is called \emph{separable} if ${\rm Lie}\,C_G(H) = \mathfrak{c}_{\mathfrak{g}}(H)$.
\end{defn}

Note that we always have ${\rm Lie}\,C_G(H) \subseteq \mathfrak{c}_{\mathfrak{g}}(H)$. In \cite{Bate-separability-TransAMS}, Bate et al.~investigated the relationship between $G$-complete reducibility and separability, and showed the following  \cite[Thm.~1.2, Thm.~1.4]{Bate-separability-TransAMS} (see~\cite{Herpel-smoothcentralizerl-trans} for more on separability).

\begin{prop}\label{p-good-separability}
Suppose that $p$ is very good for $G$. Then any subgroup of $G$ is separable in $G$.
\end{prop}
\begin{prop}\label{reductive-pair}
Suppose that $(G, M)$ is a reductive pair. Let $H$ be a subgroup of $M$ such that $H$ is a separable subgroup of $G$. If $H$ is $G$-cr, then it is also $M$-cr. 
\end{prop}
Propositions~\ref{p-good-separability} and~\ref{reductive-pair} imply that the subgroup $H$ in Theorem~\ref{main} must be non-separable, which is possible for small $p$ only.

We recap our method from~\cite{Uchiyama-Separability-JAlgebra}. Fix a maximal torus $T$ of $G=E_6$ (respectively $E_7, E_8$). Fix a system of positive roots. Let $L$ be the $A_5$ (respectively $A_6$, $A_7$)-Levi subgroup of $G$ containing $T$. Let $P$ be the parabolic subgroup of $G$ containing $L$, and let $R_u(P)$ be the unipotent radical of $P$. Let $W_L$ be the Weyl group of $L$. Abusing the notation, we write $W_L$ for the group generated by canonical representatives $n_\zeta$ of reflections in $W_L$. (See Section~2 for the definition of $n_\zeta$.) Now $W_L$ is a subgroup of $L$. 
\begin{tight_enumerate}
\item Find a subgroup $K'$ of $W_L$ acting non-separably on $R_u(P)$. 
\item If $K'$ is $G$-cr, set $K:=K'$ and go to the next step. Otherwise, add an element $t$ from the maximal torus $T$ in such a way that $K:=\langle K'\cup \{t\} \rangle$ is $G$-cr and $K$ still acts non-separably on $R_u(P)$.
\item Choose a suitable element $v\in R_u(P)$ in a $1$-dimensional curve $C$ such that $T_1(C)$ is contained in $\mathfrak{c}_{{\textup Lie}(R_u(P))}(K)$ but not contained in $\textup{Lie}(C_{R_u(P)}(K))$. Set $H:=vKv^{-1}$. Choose a connected reductive subgroup $M$ of $G$ containing $H$ such that $H$ is not $G$-cr. Show that $H$ is not $M$-cr using Proposition~\ref{GIT}.
\end{tight_enumerate}

As the first application of our construction, we consider a rationality problem for $G$-complete reducibility. We need a definition first.

\begin{defn}
Let $k_0$ be a subfield of $k$. Let $H$ be a $k_0$-defined subgroup of a $k_0$-defined reductive algebraic group $G$. Then $H$ is \emph{$G$-completely reducible over $k_0$} ($G$-cr over $k_0$ for short) if whenever $H$ is contained in a $k_0$-defined parabolic subgroup $P$ of $G$, it is contained in some $k_0$-defined Levi subgroup of $P$.
\end{defn}

Note that if $k_0$ is algebraically closed then $G$-cr over $k_0$ means $G$-cr in the usual sense. Here is the main result concerning rationality.

\begin{thm}\label{rationality}
Let $k_0$ be a nonperfect field of characteristic $2$, and let $G$ be a $k_0$-defined split simple algebraic group of type $E_n (n=6,7,8)$ of any isogeny type. Then there exists a $k_0$-defined subgroup $H$ of $G$ such that $H$ is $G$-cr but not $G$-cr over $k_0$.
\end{thm}
\begin{proof}
Use the same $H=v(a) K v(a)^{-1}$ as in the proof of Theorem~\ref{main} with $v:=v(a)$ for $a\in k_0\backslash k_0^2$. Then a similar method to~\cite[Sec.~4]{Uchiyama-Separability-JAlgebra} shows that subgroups $H$ have the desired properties. The crucial thing here is the existence of a $1$-dimensional curve $C$ such that $T_1(C)$ is contained in $\mathfrak{c}_{{\textup Lie}(R_u(P))}(K)$ but not contained in $\textup{Lie}(C_{R_u(P)}(K))$ (see~\cite[Sec.~4]{Uchiyama-Separability-JAlgebra} for details).
\end{proof}
\begin{rem}
Let $k_0$ and $G=E_6$ be as in Theorem~\ref{rationality}. Based on the construction of the $E_6$ examples in this paper, we found the first examples of nonabelian $k_0$-defined subgroups $H$ of $G$ such that $H$ is $G$-cr over $k_0$ but not $G$-cr; see~\cite{Uchiyama-Nonperfect-pre}. Note that $G$-complete reducibility over $k_0$ is invariant under central isogenies~\cite[Sec.~2]{Uchiyama-Nonperfect-pre}.
\end{rem}  

As the second application, we consider a problem concerning conjugacy classes. Given $n\in {\mathbb N}$, we let $G$ act on $G^n$ by simultaneous conjugation:
\begin{equation*}
g\cdot(g_1, g_2, \ldots, g_n) = (g g_1 g^{-1}, g g_2 g^{-1}, \ldots, g g_n g^{-1}). 
\end{equation*}
In \cite{Slodowy-book}, Slodowy proved the following result, applying Richardson's tangent space argument~\cite[Sec.~3]{Richardson-Conjugacy-Ann},~\cite[Lem.~3.1]{Richardson-orbits-BullAustralian}. 
\begin{prop}\label{conjugacy}
Let $M$ be a reductive subgroup of a reductive algebraic group $G$ defined over an algebraically closed field $k$. Let $N\in {\mathbb N}$, let $(m_1, \ldots, m_N)\in M^N$ and let $H$ be the subgroup of $M$ generated by $m_1, \ldots, m_N$. Suppose that $(G, M)$ is a reductive pair and that $H$ is separable in $G$. Then the intersection $G\cdot (m_1, \ldots, m_N)\cap M^N$ is a finite union of $M$-conjugacy classes. 
\end{prop}

Proposition~\ref{conjugacy} has many consequences; see~\cite{Bate-geometric-Inventione}, \cite{Slodowy-book}, and \cite[Sec.~3]{Vinberg-invariants-JLT} for example. Here is our main result on conjugacy classes:

\begin{thm}\label{conjugacy-counterexample}
Let $G$ be a simple algebraic group of type $E_6$ defined over an algebraically closed $k$ of characteristic $p=2$. Let $M$ be the subsystem subgroup of type $A_5 A_1$. Then there exists $N\in \mathbb{N}$ and a tuple $\mathbf{m}\in M^N$ such that $G\cdot \bold{m} \cap M^N$ is an infinite union of $M$-conjugacy classes. 
\end{thm} 
\begin{proof}
We give a sketch with one example (Section~3, Case 4). Keep the same notation $P_\lambda$, $L_\lambda$, $K$, $q_1$, $q_2$, $t$ therein. Define $K_0:=\langle K, Z(R_u(P_\lambda))\rangle$. By a standard result, there exists a finite subset $F=\{z_1,\cdots, z_n\}$ of $Z(R_u(P_\lambda))$ such that $C_{P_\lambda}(\langle K\cup F \rangle)=C_{R_u(P_\lambda)}(K_0)$. Let $\textbf{m}:=(q_1, q_2, t, z_1,\cdots, z_n)$. Set $N:=n+3$. Then, a similar computation to that of~\cite[Lems.~5.1, 5.2]{Uchiyama-Separability-JAlgebra} shows that $G\cdot\mathbf{m}\cap P_\lambda(M)^N$ is an infinite union of $P_\lambda(M)$-conjugacy classes. (Here, the existence of a $1$-dimensional curve $C$ such that $T_1(C)$ is contained in $\mathfrak{c}_{{\textup Lie}(R_u(P))}(K)$ but not contained in $\textup{Lie}(C_{R_u(P)}(K))$ is crucial.) 

Let $c_\lambda:P_\lambda\rightarrow L_\lambda$ be the canonical projection. Then $c_\lambda((q_1, q_2, t, z_1,\cdots, z_n))=(q_1, q_2, t)$ and an easy computation by Magma shows that $K=\langle q_1, q_2, t \rangle$ is $L_\lambda$-ir. (This is easy to check since $L_\lambda$ is of type $A_5$.) Now the same argument as that in the proof of~\cite[Prop.~3.5.2]{Stewart-thesis} shows that $\mathbf{m}$ has the desired property. 
\end{proof}

\begin{rem}
The following was pointed out by the referee: our previous result~\cite[Lem.~5.3]{Uchiyama-Separability-JAlgebra} was wrong and a counterexample was given in~\cite[Ex.~4.22]{Martin-Lond-Kulshammer-Arx}. 
A direct computation shows that~\cite[Thm.~1.12]{Uchiyama-Separability-JAlgebra} (which depends on~\cite[Lem.~5.3]{Uchiyama-Separability-JAlgebra}) is also wrong. The point is that the subgroup $K$ there is not $L_\lambda$-ir, thus the second part of the proof of Theorem~\ref{conjugacy-counterexample} does not go through. Likewise using subgroups $K$ of $G$ of type $E_7$ and $E_8$ in Sections 4 and 5 we can find tuples $\textbf{m}$ for $G$ of type $E_7$ and $E_8$ such that $G\cdot\mathbf{m}\cap P_\lambda(M)^N$ is an infinite union of $P_\lambda(M)$-conjugacy classes, but these subgroups $K$ are not $L_\lambda$-ir. A direct computation shows that our method does not generate $\textbf{m}$ with the desired property in these cases.
\end{rem}

Now we discuss another application of our construction with a different flavor. Here, we consider a question of K\"ulshammer on representations of finite groups in reductive algebraic groups. Let $\Gamma$ be a finite group. By a representation of $\Gamma$ in a reductive algebraic group $G$, we mean a homomorphism from $\Gamma$ to $G$. We write $\textup{Hom}(\Gamma, G)$ for the set of representations $\rho$ of $\Gamma$ in $G$. The group $G$ acts on $\textup{Hom}(\Gamma, G)$ by conjugation. Let $\Gamma_p$ be a Sylow $p$-subgroup of $G$. In \cite[Sec.~2]{Kulshammer-Donovan-Israel}, K\"ulshammer asked: 

\begin{question}\label{KulshammerQ}
Let $G$ be a reductive algebraic group defined over an algebraically closed field of characteristic $p$. Let $\rho_p \in \textup{Hom}(\Gamma_p, G)$. Then are there only finitely many representations $\rho \in \textup{Hom}(\Gamma, G)$ such that $\rho\mid_{\Gamma_p}$ is $G$-conjugate to $\rho_p$? 
\end{question}

In~\cite{Bate-QuestionOfKulshammer}, Bate et al. presented an example where $p=2, G=G_2$ and $G$ has a finite subgroup $\Gamma$ with Sylow $2$-subgroup $\Gamma_2$ such that $\Gamma$ has an infinite family of pairwise non-conjugate representations $\rho$ whose restrictions to $\Gamma_2$ are all conjugate. In this paper, we present another example which answers Question~\ref{KulshammerQ} negatively:

\begin{thm}\label{KulshammerA}
Let $G$ be a simple simply-connected algebraic group of type $E_6$ defined over an algebraically closed field $k$ of characteristic $p=2$.  Then there exist a finite group $\Gamma$ with a Sylow $2$-subgroup $\Gamma_2$ and representations $\rho_a\in \textup{Hom}(\Gamma, G)$ for $a\in k$ such that $\rho_a$ is not conjugate to $\rho_b$ for $a\neq b$ but the restrictions $\rho_a\mid_{\Gamma_2}$ are pairwise conjugate for all $a\in k$. 
\end{thm}

Note that the example of Theorem~\ref{KulshammerA} is derived from Case 4 in the proof of Theorem~\ref{main}. We also present an example giving a negative answer to Question~\ref{KulshammerQ} for a non-connected reductive $G$ (this is much easier than the connected case):

\begin{thm}\label{KulshammerB}
Let $k$ be an algebraically closed field of characteristic $2$. Let $G:=SL_3(k)\rtimes \langle \sigma \rangle$ where $\sigma$ is the nontrivial graph automorphism of $SL_3(k)$. Let $d\geq 3$ be odd. Let $D_{2d}$ be the dihedral group of order $2d$. Let $$\Gamma:=D_{2d}\times C_2 =\langle r,s,z\mid r^d=s^2=z^2=1, s r s^{-1} = r^{-1}, [r,z]=[s,z]=1\rangle.$$ Let $\Gamma_2=\langle s,z \rangle$ (a Sylow $2$-subgroup of $\Gamma$). Then there exist representations $\rho_a\in \textup{Hom}(\Gamma, G)$ for $a\in k$ such that 
$\rho_a$ is not conjugate to $\rho_b$ for $a\neq b$ but restrictions $\rho_a\mid_{\Gamma_2}$ are pairwise conjugate for all $a\in k$.
\end{thm}

Here is the structure of this paper. In Section 2, we set out the notation and give a few preliminary results. Then in Section 3, 4, 5, we present a list of $G$-cr but non $M$-cr subgroups for $G=E_6, E_7, E_8$ respectively. This proves Theorem~\ref{main}. Some details of our method will be explained in Section 3 using one of the examples for $G=E_6$. Finally in Section 6, we give proofs of Theorems~\ref{KulshammerA} and~\ref{KulshammerB}.

\section{Preliminaries}
Throughout, we denote by $k$ an algebraically closed field of positive characteristic $p$. Let $G$ be an algebraic group defined over $k$. We write $R_u(G)$ for the \emph{unipotent radical} of $G$, and $G$ is called (possibly non-connected) \emph{reductive} if $R_u(G)=\{1\}$. In particular, $G$ is \emph{simple} as an algebraic group if $G$ is connected and all proper normal subgroups of $G$ are finite. In this paper, when a subgroup $H$ of $G$ acts on $G$, we assume $H$ acts on $G$ by inner automorphisms. We write $C_G(H)$ and $\mathfrak{c}_{\mathfrak{g}}(H)$ for the global and the infinitesimal centralizers of $H$ in $G$ and $\mathfrak{g}$ respectively. We write $X(G)$ and $Y(G)$ for the set of characters and cocharacters of $G$ respectively. 

Let $G$ be a connected reductive algebraic group. Fix a maximal torus $T$ of $G$. Let $\Psi(G,T)$ denote the set of roots of $G$ with respect to $T$. We sometimes write $\Psi(G)$ for $\Psi(G,T)$. Let $\zeta\in\Psi(G)$. We write $U_\zeta$ for the corresponding root subgroup of $G$ and $\mathfrak{u}_\zeta$ for the Lie algebra of $U_\zeta$. We define $G_\zeta := \langle U_\zeta, U_{-\zeta} \rangle$. Let $\zeta, \xi \in \Psi(G)$. Let $\xi^{\vee}$ be the coroot corresponding to $\xi$. Then $\zeta\circ\xi^{\vee}:k^{*}\rightarrow k^{*}$ is a homomorphism such that $(\zeta\circ\xi^{\vee})(a) = a^n$ for some $n\in\mathbb{Z}$. We define $\langle \zeta, \xi^{\vee} \rangle := n$.
Let $s_\xi$ denote the reflection corresponding to $\xi$ in the Weyl group of $G$. Each $s_\xi$ acts on the set of roots $\Psi(G)$ by the following formula~\cite[Lem.~7.1.8]{Springer-book}:
$
s_\xi\cdot\zeta = \zeta - \langle \zeta, \xi^{\vee} \rangle \xi. 
$
\noindent By \cite[Prop.~6.4.2, Lem.~7.2.1]{Carter-simple-book} we can choose homomorphisms $\epsilon_\zeta : k \rightarrow U_\zeta$  so that 
$
n_\xi \epsilon_\zeta(a) n_\xi^{-1}= \epsilon_{s_\xi\cdot\zeta}(\pm a)
            \text{ where } n_\xi = \epsilon_\xi(1)\epsilon_{-\xi}(-1)\epsilon_{\xi}(1).  \label{n-action on group}
$
We define $e_\zeta:=\epsilon_\zeta'(0)$.

We recall \cite[Sec.~2.1--2.3]{Richardson-conjugacy-Duke} for the characterization of a parabolic subgroup $P$ of $G$, a Levi subgroup $L$ of $P$, and the unipotent radical $R_u(P)$ of $P$ in terms of a cocharacter of $G$ and state a result from GIT (Proposition~\ref{GIT}).  

\begin{defn}
Let $X$ be an affine variety. Let $\phi : k^*\rightarrow X$ be a morphism of algebraic varieties. We say that $\displaystyle\lim_{a\rightarrow 0}\phi(a)$ exists if there exists a morphism $\hat\phi: k\rightarrow X$ (necessarily unique) whose restriction to $k^{*}$ is $\phi$. If this limit exists, we set $\displaystyle\lim_{a\rightarrow 0}\phi(a) = \hat\phi(0)$.
\end{defn}

\begin{defn}
Let $\lambda$ be a cocharacter of $G$. Define
$
P_\lambda := \{ g\in G \mid \displaystyle\lim_{a\rightarrow 0}\lambda(a)g\lambda(a)^{-1} \text{ exists}\}, $\\
$L_\lambda := \{ g\in G \mid \displaystyle\lim_{a\rightarrow 0}\lambda(a)g\lambda(a)^{-1} = g\}, \,
R_u(P_\lambda) := \{ g\in G \mid  \displaystyle\lim_{a\rightarrow0}\lambda(a)g\lambda(a)^{-1} = 1\}. 
$
\end{defn}
Note that $P_\lambda$ is a parabolic subgroup of $G$, $L_\lambda$ is a Levi subgroup of $P_\lambda$, and $R_u(P_\lambda)$ is the unipotent radical of $P_\lambda$~\cite[Sec.~2.1-2.3]{Richardson-conjugacy-Duke}. By~\cite[Prop.~8.4.5]{Springer-book}, any parabolic subgroup $P$ of $G$, any Levi subgroup $L$ of $P$, and any unipotent radical $R_u(P)$ of $P$ can be expressed in this form. It is well known that $L_\lambda = C_G(\lambda(k^*))$. 

Let $M$ be a reductive subgroup of $G$. There is a natural inclusion $Y(M)\subseteq Y(G)$ of cocharacter groups. Let $\lambda\in Y(M)$. We write $P_\lambda(G)$ or just $P_\lambda$ for the parabolic subgroup of $G$ corresponding to $\lambda$, and $P_\lambda(M)$ for the parabolic subgroup of $M$ corresponding to $\lambda$. It is obvious that $P_\lambda(M) = P_\lambda(G)\cap M$ and $R_u(P_\lambda(M)) = R_u(P_\lambda(G))\cap M$. 

\begin{defn}\label{homomorphism}
Let $\lambda\in Y(G)$. Define a map $c_\lambda : P_\lambda \rightarrow L_\lambda$ by 
$
c_\lambda(g) := \displaystyle\lim_{a\rightarrow 0} \lambda(a) g \lambda(a)^{-1}.
$
\end{defn}
Note that the map $c_\lambda$ is the usual canonical projection from $P_\lambda$ to $L_\lambda \cong P_\lambda / R_u(P_\lambda)$.
Now we state a result from GIT (see~\cite[Lem.~2.17, Thm.~3.1]{Bate-geometric-Inventione},~\cite[Thm.~3.3]{Bate-uniform-TransAMS}).
\begin{prop}\label{GIT}
Let $H$ be a subgroup of $G$. Let $\lambda$ be a cocharacter of $G$ with $H\subseteq P_\lambda$. If $H$ is $G$-cr, there exists $v\in R_u(P_\lambda)$ such that $c_\lambda(h) = vhv^{-1}$ for every $h\in H$. 
\end{prop}

\section{The $E_6$ examples}
 For the rest of the paper, we assume $k$ is an algebraically closed field of characteristic $2$. Let $G$ be a simple algebraic group of type $E_6$ defined over $k$. Without loss, we assume that $G$ is simply-connected. Fix a maximal torus $T$ of $G$. Pick a Borel subgroup $B$ of $G$ containing $T$.
Let $\Sigma = \{ \alpha,\beta,\gamma,\delta,\epsilon,\sigma \}$ be the set of simple roots of $G$ corresponding to $B$ and $T$. The next figure defines how each simple root of $G$ corresponds to each node in the Dynkin diagram of $E_6$.
\begin{figure}[h]
                \centering
                \scalebox{0.7}{
                \begin{tikzpicture}
                \draw (0,0) to (1,0);
                \draw (1,0) to (2,0);
                \draw (2,0) to (3,0);
                \draw (3,0) to (4,0);
                \draw (2,0) to (2,1);
                \fill (0,0) circle (1mm);
                \fill (1,0) circle (1mm);
                \fill (2,0) circle (1mm);
                \fill (3,0) circle (1mm);
                \fill (4,0) circle (1mm);
                \fill (2,1) circle (1mm);
                \draw[below] (0,-0.3) node{$\alpha$};
                \draw[below] (1,-0.3) node{$\beta$};
                \draw[below] (2,-0.3) node{$\gamma$};
                \draw[below] (3,-0.3) node{$\delta$};
                \draw[below] (4,-0.3) node{$\epsilon$};
                \draw[right] (2.3,1) node{$\sigma$};
                \end{tikzpicture}}
\end{figure}
We label the positive roots of $G$ as shown in Table \ref{E6} in the Appendix~\cite[Appendix, Table B]{Freudenthal-Lie-book}. Define
$L: = \langle T, G_{22},\cdots, G_{36} \rangle$, 
$P: = \langle L, U_{1},\cdots, U_{21} \rangle, W_L:=\langle n_\alpha, n_\beta, n_\gamma, n_\delta, n_\epsilon \rangle$.
Then $P$ is a parabolic subgroup of $G$, $L$ is a Levi subgroup of $P$, and $\Psi(R_u(P)) = \{ 1,\cdots,21 \}$. Let $M=\langle L, G_{21} \rangle$. Then $M$ is a subsystem subgroup of type $A_5A_1$, $(G,M)$ is a reductive pair, and $\Psi(M)=\{\pm 21,\cdots, \pm 36\}$. Note that $L$ is generated by $T$ and all root subgroups with $\sigma$-weight $0$, and $M$ is generated by $L$ and all root subgroups with  $\sigma$-weight $\pm 2$. Here, by the $\sigma$-weight of a root subgroup $U_\zeta$, we mean the $\sigma$-coefficient of $\zeta$.  

Using Magma, we found that there are $56$ subgroups of $W_L$ up to conjugacy, and $11$ of them act non-separably on $R_u(P)$. Table~\ref{E6examples} lists these $11$ subgroups $K'$, and also gives the choice of $t$ we use to give $K:=\langle K' \cup \{t\}\rangle$. Note that $[L,L]=SL_6$ since $G$ is simply-connected. We identify $n_\alpha$, $n_\beta$, $n_\gamma$, $n_\delta$, $n_\epsilon$ with $(1 2)$, $(2 3)$, $(3 4)$, $(4 5)$, $(5 6)$ in $S_6$. To illustrate our method, we look at Case $4$ closely.

\begin{table}[!h]
\centering
\scalebox{0.8}{
\begin{tabular}{|l|l|l|l|l|}
\hline
 case & generators of $K'$ & $|K'|$ & $t$ & $v(a)$ \\
\hline
1 & (1 5)(2 3)(4 6) & 2 & $(\alpha^{\vee}+\epsilon^{\vee})(b)$ & $\epsilon_{7}(a)\epsilon_{8}(a)$\\
2 & (1 5)(4 6), (1 4 5 6)(2 3) &4 & $\alpha^{\vee}(b)$ & $\epsilon_{10}(a)\epsilon_{13}(a)$\\
3 & (2 4)(3 6), (1 5)(2 6)(3 4)&4 & $(\alpha^{\vee}+\epsilon^{\vee})(b)$ & $\epsilon_{7}(a)\epsilon_{8}(a)$ \\
4 & (1 5)(2 3)(4 6), (1 4 2)(3 6 5)& 6& $(\alpha^{\vee}+\epsilon^{\vee})(b)$ & $\epsilon_{7}(a)\epsilon_{8}(a)$\\
5& (1 5)(2 6)(3 4), (1 4 2)(3 6 5)& 6& $(\alpha^{\vee}+\epsilon^{\vee})(b) $ & $\epsilon_{7}(a)\epsilon_{8}(a)$ \\
6& (4 6), (1 4)(2 3)(5 6), (1 5)(4 6)& 8& $\alpha^{\vee}(b)$ & $\epsilon_{10}(a)\epsilon_{13}(a)$\\
7& (1 5)(2 6)(3 4), (2 4)(3 6), (1 2 4)(3 5 6)& 12& $(\alpha^{\vee}+\epsilon^{\vee})(b)$ & $\epsilon_{7}(a)\epsilon_{8}(a)$\\
8& (1 4)(2 3)(5 6), (1 3 5)(2 4 6), (2 4 6)& 18 & $(\alpha^{\vee}+\beta^{\vee})(b)$ &$\epsilon_{11}(a)\epsilon_{12}(a)$\\
9& (1 4)(2 3)(5 6), (3 5)(4 6), (1 3 5), (2 4 6)& 36&  $(\alpha^{\vee}+\beta^{\vee})(b)$ &$\epsilon_{11}(a)\epsilon_{12}(a)$\\
10& (1 4 5 6)(2 3), (3 5)(4 6), (1 3 5), (2 4 6)& 36 & $(\alpha^{\vee}+\beta^{\vee})(b)$ &$\epsilon_{11}(a)\epsilon_{12}(a)$\\
11& (1 3), (1 4)(2 3)(5 6), (1 3)(4 6), (1 5 3), (2 6 4)& 72&  $(\alpha^{\vee}+\beta^{\vee})(b)$ &$\epsilon_{11}(a)\epsilon_{12}(a)$\\
\hline 
\end{tabular}
}
\caption{The $E_6$ examples}\label{E6examples}
\end{table}

\noindent $\bullet$ Case $4$: \\

Let $b\in k$ such that $b^3=1$ and $b\neq 1$. Define
\begin{alignat*}{2}
q_1 &:= n_\alpha n_\beta n_\gamma n_\beta n_\alpha n_\beta n_\gamma n_\beta n_\gamma n_\delta n_\epsilon n_\delta n_\gamma n_\epsilon, \; q_2 := n_\alpha n_\beta n_\gamma n_\delta n_\gamma n_\beta n_\alpha n_\beta n_\delta n_\epsilon n_\delta, \\
t &:=(\alpha^{\vee}+\epsilon^{\vee})(b), \;
K' := \langle q_1, q_2 \rangle, \;
K := \langle q_1, q_2, t \rangle.
\end{alignat*}  

It is easy to calculate how $W_L$ acts on $\Psi(R_u(P))$. Let $\pi: W_L \rightarrow \textup{ Sym }(\Psi(R_u(P)))\cong S_{21}$ be the corresponding homomorphism. Then we have 
\begin{alignat*}{2}
\pi(q_1) &= (1\; 5\; 4)(2\; 3\; 6)(9\; 12\; 10)(11\; 13\; 14)(15\; 16\; 17)(18\; 20\; 19),\\
\pi(q_2) &= (1\; 2)(3\; 4)(5\; 6)(7\; 8)(9\; 14)(10\; 11)(12\; 13)(15\; 18)(16\; 19)(17\; 20).
\end{alignat*}
The orbits of $\langle q_1, q_2\rangle$ are $O_1=\{1,2,3,4,5,6\}$, $O_7=\{7,8\}$, $O_9=\{9,10,11,12,13,14\}$, $O_{15}=\{15,16,17,18,19,20\}$, $O_{21}=\{21\}$. Since $t$ acts trivially on $e_7+e_8$, \cite[Lem.~2.8]{Uchiyama-Separability-JAlgebra} yields  
\begin{prop}\label{Liealgebra}
$e_{7}+e_{8}\in \mathfrak{c}_{\textup{Lie}\left(R_u(P)\right)}(K)$. 
\end{prop}
\begin{prop}\label{positive proposition}
Let $u\in C_{R_u(P_{\alpha\beta\gamma\delta\epsilon\eta})}(K)$. Then $u$ must have the form, 
\begin{equation*}
u=\prod_{i=1}^{6}\epsilon_i(a)\prod_{i=7}^{8}\epsilon_i(b)\prod_{i=9}^{14}\epsilon_i(c)\left(\prod_{i=15}^{20}
\epsilon_i(a+b+c)\right)\epsilon_{21}(a_{21}) \textup{ for some } a,b,c,a_{21}\in k.  
\end{equation*}  
\end{prop}
\begin{proof}
By~\cite[Prop.~8.2.1]{Springer-book}, $u$ can be expressed uniquely as 
$
u = \prod_{i = 1}^{21} \epsilon_i(a_i) \textup{ for some } a_i\in k.
$
Since $p=2$ we have
$
n_\xi \epsilon_\zeta(a) n_\xi^{-1}= \epsilon_{s_\xi\cdot\zeta}(a) \textup{ for any } a\in k \textup{ and }\xi, \zeta \in \Psi(G). 
$ 
Then a calculation using the commutator relations (\cite[Lem.~32.5, Lem.~33.3]{Humphreys-book1}) shows that
\begin{alignat}{2}
q_2 u q_2^{-1}=&\; \epsilon_1(a_2)\epsilon_2(a_1)\epsilon_3(a_4)\epsilon_4(a_3)
\epsilon_5(a_6)\epsilon_6(a_5)\epsilon_7(a_8)\epsilon_8(a_7)\epsilon_9(a_{14})\epsilon_{10}(a_{11})\epsilon_{11}(a_{10})
\epsilon_{12}(a_{13})\epsilon_{13}(a_{12})\nonumber\\
&\; \epsilon_{14}(a_{9}) \epsilon_{15}(a_{18})\epsilon_{16}(a_{19})\epsilon_{17}(a_{20})\epsilon_{18}(a_{15})
\epsilon_{19}(a_{16})\epsilon_{20}(a_{17})\epsilon_{21}(a_{18}+a_{21}). \label{reordered}
\end{alignat}
\noindent Since $q_1$ and $q_2$ centralize $u$, we have 
$
a_1 = \cdots = a_6, \; a_7 = a_{8}, \; a_{9} = \cdots = a_{14}, \; a_{15} = \cdots = a_{20}.  
$
Set
$
a_1 = a, \; a_7 = b, \; a_{9} = c, \; a_{15} = d. 
$
Then (\ref{reordered}) simplifies to 
\begin{alignat*}{2}
q_2 u q_2^{-1} & =\prod_{i=1}^{6}\epsilon_i(a)\prod_{i=7}^{8}\epsilon_i(b)\prod_{i=9}^{14}\epsilon_i(c)
\left(\prod_{i=15}^{20}\epsilon_i(d)\right)\epsilon_{21}(a^2+b^2+c^2+d^2+a_{21}). 
\end{alignat*}
Since $q_2$ centralizes $u$, comparing the arguments of the $\epsilon_{21}$ term on both sides, we must have
\begin{equation*}
a_{21} = a^2 + b^2 + c^2 + d^2 + a_{21},
\end{equation*}
which is equivalent to
$
a + b + c + d = 0.
$
Then we obtain the desired result. 
\end{proof}

\begin{prop}
$K$ acts non-separably on $R_u(P)$.
\end{prop}
\begin{proof}
Proposition~\ref{positive proposition} and a similar argument to that of the proof of~\cite[Prop.~3.3]{Uchiyama-Separability-JAlgebra} show that $e_{7}+e_{8}\notin \textup{Lie}\,C_{R_u(P)}(K)$. Then Proposition~\ref{Liealgebra} gives the desired result.
\end{proof}

\begin{rem}\label{important}
The following three facts are essential for the argument above:
\begin{tight_enumerate}
\item The orbit $O_7$ contains a pair of roots corresponding to a non-commuting pair of root subgroups which get swapped by $q_2$; $q_2\cdot(\epsilon_{7}(a)\epsilon_{8}(a))=\epsilon_{8}(a)\epsilon_{7}(a)=\epsilon_{7}(a)\epsilon_{8}(a)\epsilon_{21}(a^2)$.
\item The correction term $\epsilon_{21}(a^2)$ in the last equation is contained in $Z(R_u(P))$. 
\item The root $21$ corresponding to the correction term is fixed by $\pi(q_2)$.
\end{tight_enumerate}
\end{rem}

Now, let
$C:= \left\{\prod_{i=7}^{8} \epsilon_i(a) \mid a\in k \right\}$, pick any $a\in k^*$, and let $v(a):=\prod_{i=7}^{8}\epsilon_i(a)$. Now set 
$H := v(a) K v(a)^{-1} = \langle q_1, q_2 \epsilon_{21}(a^2), t \rangle.$
Note that $H \subset M, H\not\subset L$. 

\begin{prop}\label{non-M-cr}
$H$ is not $M$-cr.
\end{prop}
\begin{proof}
Let $\lambda=\alpha^{\vee}+2\beta^{\vee}+3\gamma^{\vee}+2\delta^{\vee}+\epsilon^{\vee}+2{\sigma}^{\vee}$. 
Then $L=L_\lambda$, $P=P_\lambda$. Let $c_\lambda : P_\lambda \rightarrow L_\lambda$ be the homomorphism from Definition~\ref{homomorphism}. 
In order to prove that $H$ is not $M$-cr, by Proposition~\ref{GIT} it suffices to find a tuple $(h_1, h_2)\in H^2$ that is not  $R_u\left(P_\lambda(M)\right)$-conjugate to $c_\lambda\left((h_1, h_2)\right)$. Set 
$
h_1 := v(a) q_1 v(a)^{-1}, \; h_2 := v(a) q_2 v(a)^{-1}. 
$
Then
\begin{equation*}
c_\lambda\left((h_1, h_2)\right) = \lim_{x\rightarrow 0}\left(\lambda(x)q_1 \lambda(x)^{-1}, \; \lambda(x)q_2 \epsilon_{21}(a^2)\lambda(x)^{-1}\right)=(q_1, q_2).
\end{equation*}   
Now suppose that $(h_1, h_2)$ is $R_u\left(P_\lambda(M)\right)$-conjugate to $c_\lambda\left((h_1, h_2)\right)$. Then there exists $m\in R_u\left(P_\lambda(M)\right)$ such that
$
 m v(a) q_1 v(a)^{-1}m^{-1} = q_1, \,
 m v(a) q_2 v(a)^{-1}m^{-1} = q_2. 
$
Thus we have
$
m v(a) \in C_{R_u(P_\lambda)}(K). 
$
Note that 
$
\Psi\left(R_u\left(P_\lambda(M)\right)\right) = \{21\}. 
$
Let  
$
m = \epsilon_{21}(a_{21}) \textup{ for some }a_{21}\in k.
$
Then we have
$
 m v(a) = \epsilon_{7}(a) \epsilon_{8}(a) \epsilon_{21}(a_{21}) \in C_{R_u(P_\lambda)}(K). 
$
This contradicts Proposition~\ref{positive proposition}.
\end{proof}

\begin{prop}\label{G-cr}
$H$ is $G$-cr.
\end{prop}
\begin{proof}
Since $H$ is $G$-conjugate to $K$, it is enough to show that $K$ is $G$-cr. Since $K$ is contained in $L$, by~\cite[Prop.~3.2]{Serre-building} it suffices to show that $K$ is $L$-cr. Then by~\cite[Lem.~2.12]{Bate-geometric-Inventione}, it is enough to show that $K$ is $[L,L]$-cr. Note that $[L,L]=SL_6$. An easy matrix computation shows that $K$ acts semisimply on $k^n$, so $K$ is $G$-cr by~\cite[Ex.~3.2.2(a)]{Serre-building}.
\end{proof}  

It is clear that similar arguments work for the other cases. We omit proofs.

\section{The $E_7$ examples}
Let $G$ be a simple simply-connected algebraic group of type $E_7$ defined over $k$. Fix a maximal torus $T$ of $G$, and a Borel subgroup of $G$ containing $T$. We define the set of simple roots $\Sigma=\{\alpha, \beta, \gamma, \delta, \epsilon, \eta, \sigma \}$ as in the following Dynkin diagram. The positive roots of $G$ are listed in~\cite[Appendix, Table B]{Freudenthal-Lie-book}.
\newline
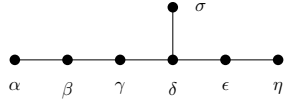
\begin{figure}[h!]
                \centering
                \scalebox{0.7}{
                \begin{tikzpicture}
                \draw (1,0) to (2,0);
                \draw (2,0) to (3,0);
                \draw (3,0) to (4,0);
                \draw (4,0) to (5,0);
                \draw (5,0) to (6,0);
                \draw (4,0) to (4,1);
                \fill (1,0) circle (1mm);
                \fill (2,0) circle (1mm);
                \fill (3,0) circle (1mm);
                \fill (4,0) circle (1mm);
                \fill (5,0) circle (1mm);
                \fill (6,0) circle (1mm);
                \fill (4,1) circle (1mm);
                \draw[below] (1,-0.3) node{$\alpha$};
                \draw[below] (2,-0.3) node{$\beta$};
                \draw[below] (3,-0.3) node{$\gamma$};
                \draw[below] (4,-0.3) node{$\delta$};
                \draw[below] (5,-0.3) node{$\epsilon$};
                \draw[below] (6,-0.3) node{$\eta$};
                \draw[right] (4.3,1) node{$\sigma$};
                \end{tikzpicture}}
\end{figure}

Let $L$ be the subgroup of $G$ generated by $T$ and all root subgroups of $G$ with $\sigma$-weight $0$. Let $P$ be the subgroup of $G$ generated by $L$ and all root subgroups of $G$ with $\sigma$-weight $1$ or $2$. Then $P$ is a parabolic subgroup of $G$ and $L$ is a Levi subgroup of $P$. Let $W_L:=\langle n_\alpha, n_\beta, n_\gamma, n_\delta, n_\epsilon, n_\eta \rangle$. Let $M$ be the subgroup of $G$ generated by $L$ and all root subgroups of $G$ with $\sigma$-weight $\pm2$. Then $M$ is the subsystem subgroup of $G$ of type $A_7$, and $(G,M)$ is a reductive pair.  

In the $E_7$ cases, we take $t=1$ and $K':=K$; so each $K$ is a subgroup of $W_L$. We use the same method as the $E_6$ examples, so we just give a sketch.

Using Magma, we found $95$ non-trivial subgroups $K$ of $W_L$ up to conjugacy, and $19$ of them are $G$-cr. Only two of them act non-separably on $R_u(P)$ (see Table 2). We determined $G$-complete reducibility and non-separability of $K$ by a similar argument to that of the proof of Proposition~\ref{G-cr}. Note that $[L,L]=SL_7$. We identify $n_\alpha$, $\cdots$, $n_\eta$ with $(1 2),\cdots, (6 7)$ in $S_7$. 

\begin{table}[!h]
\centering
\scalebox{0.7}{
\begin{tabular}{|l|l|l|l|l|l|}
\hline
 case & generators of $K$ & $|K|$  \\
\hline
 1 & (2 5)(3 7)(4 6), (1 4 3 2 5 7 6) & 14 \\
 2 & (2 6 7)(3 5 4), (2 5)(3 7)(4 6), (1 6 7 5 2 3 4)& 42 \\
\hline
\end{tabular}
}
\caption{The $E_7$ examples}\label{E7examples}
\end{table}

\noindent$\bullet$ Case $1$ was in~\cite[Sec.~3]{Uchiyama-Separability-JAlgebra}.\\
$\bullet$ Case $2$: 

Let $q_1 = n_\epsilon n_\gamma n_\alpha, q_2 = n_\alpha n_\gamma n_\alpha n_\beta n_\gamma n_\alpha n_\beta n_\gamma n_\eta n_\epsilon n_\delta n_\gamma n_\beta$,
$K =\langle q_1, q_2 \rangle\cong \textup{Frob}_{42}$ (Frobenius group of order $42$). 
We label some roots of $G$ in Table~\ref{Case14} in Appendix. It can be calculated that $K$ has an orbit $\{1,\cdots, 14\}$ which contains only one non-commuting pair of roots $\{2, 10\}$ contributing to a correction term that lies in $U_{15}$. Also, $\pi(q_1)$ swaps $2$ with $10$, and fixes $15$. Thus $K$ acts non-separably on $R_u(P)$ (see Remark~\ref{important}). Now, set $v(a)=\prod_{i=1}^{14}\epsilon_{i}(a)$, and $H:=v(a)\cdot K$. Then a similar argument to that of the proof of Proposition~\ref{non-M-cr} show that $H$ is not $M$-cr.

\section{The $E_8$ examples}
Let $G$ be a simple simply-connected algebraic group of type $E_8$ defined over $k$. Fix a maximal torus $T$ and a Borel subgroup $B$ containing $T$. Define $\Sigma =\{\alpha, \beta, \gamma, \delta, \epsilon, \eta, \xi, \sigma\}$ by the next Dynkin diagram. All roots of $G$ are listed in~\cite[Appendix, Table B]{Freudenthal-Lie-book}.
\begin{figure}[h]
                \centering
                \scalebox{0.7}{
                \begin{tikzpicture}
                \draw (0,0) to (1,0);
                \draw (1,0) to (2,0);
                \draw (2,0) to (3,0);
                \draw (3,0) to (4,0);
                \draw (4,0) to (5,0);
                \draw (5,0) to (6,0);
                \draw (4,0) to (4,1);
                \fill (0,0) circle (1mm);
                \fill (1,0) circle (1mm);
                \fill (2,0) circle (1mm);
                \fill (3,0) circle (1mm);
                \fill (4,0) circle (1mm);
                \fill (5,0) circle (1mm);
                \fill (6,0) circle (1mm);
                \fill (4,1) circle (1mm);
                \draw[below] (0,-0.3) node{$\alpha$};
                \draw[below] (1,-0.3) node{$\beta$};
                \draw[below] (2,-0.3) node{$\gamma$};
                \draw[below] (3,-0.3) node{$\delta$};
                \draw[below] (4,-0.3) node{$\epsilon$};
                \draw[below] (5,-0.3) node{$\eta$};
                \draw[below] (6,-0.3) node{$\xi$};
                \draw[right] (4.3,1) node{$\sigma$};
                \end{tikzpicture}}
\end{figure}
Let $L$ be the subgroup of $G$ generated by $T$ and all root subgroups of $G$ with $\sigma$-weight $0$. Let $P$ be the subgroup of $G$ generated by $L$ and all root subgroups of $G$ with $\sigma$-weight $1$, $2$, or $3$. Let $W_L:=\langle n_\alpha, n_\beta, n_\gamma, n_\delta, n_\epsilon, n_\eta, n_\xi \rangle.$ Then $P$ is a parabolic subgroup of $G$, and $L$ is a Levi subgroup of $P$.
Let $M$ be the subgroup of $G$ generated by $L$ and all root subgroups of $G$ with $\sigma$-weight $\pm2$. Then $M$ is a subsystem subgroup of type $D_8$, and $(G,M)$ is a reductive pair. In the $E_8$ cases, we take $t=1$ and $K':=K$; so each $K$ is a subgroup of $W_L$. We use the same method as in the $E_6, E_7$ examples, so we just give a sketch.

With Magma, we found $295$ non-trivial subgroups $K$ of $W$ up to conjugacy, and $31$ of them are $G$-cr. Only two of them act non-separably on $R_u(P)$ (see Table 3). Note that $[L,L]\cong SL_8$. We identify $n_\alpha, \cdots, n_\xi$ with $(1 2), \cdots, (7 8)$ in $S_8$. 
\begin{table}[!h]
\centering
\scalebox{0.7}{
\begin{tabular}{|l|l|l|l|l|l|}
\hline
 case & generators of $K$ & $|K|$  \\
\hline
 1 & (2 6)(4 5)(7 8), (1 4 2 8 7 6 5) & 14 \\
 2 & (1 7 5)(2 6 8), (1 2)(5 8)(6 7), (1 2 7 5 4 8 6) &42\\
\hline 
\end{tabular}
}
\caption{The $E_8$ examples}\label{E8examples}
\end{table}

\noindent$\bullet$ Case 1:

Let 
$
q_1 = n_\beta n_\gamma n_\delta n_\epsilon n_\delta n_\gamma n_\beta n_\delta n_\xi, \;
q_2 = n_\alpha n_\beta n_\gamma n_\beta n_\alpha n_\beta n_\delta n_\epsilon n_\eta n_\xi n_\eta n_\epsilon n_\delta n_\gamma n_\beta n_\xi n_\eta n_\epsilon,\\
K=\langle q_1, q_2 \rangle.
$

We label some roots of $G$ as in Table~\ref{Case9} in the Appendix. It can be calculated that $K$ has an orbit $O_{1}=\{1,\cdots,7\}$ which contains only one non-commuting pair of roots $\{3, 4\}$, contributing a correction term that lies in  $U_8$. Also $\pi(q_1)$ swaps $3$ with $4$, and fixes $8$. So $K$ acts nonseparably on $R_u(P)$ (see Remark~\ref{important}). Now let $v(a)=\prod_{i=1}^{7}\epsilon_i(a)$, and define $H=v(a)\cdot K$. Then it is clear that 
$H$ is not $M$-cr by the same argument as in the $E_6$ cases. 

\noindent$\bullet$ Case 2:

Let
$
q_1 = n_\alpha n_\beta n_\gamma n_\delta n_\epsilon n_\eta n_\epsilon n_\delta n_\gamma n_\beta n_\alpha n_\epsilon n_\eta n_\epsilon n_\beta n_\gamma n_\delta n_\epsilon n_\delta n_\gamma n_\beta n_\eta n_\xi n_\eta,\\
q_2 = n_\alpha n_\epsilon n_\eta n_\xi n_\eta n_\epsilon n_\eta, 
q_3 = n_\alpha n_\beta n_\gamma n_\delta n_\epsilon n_\eta n_\epsilon n_\delta n_\gamma n_\beta n_\epsilon n_\xi n_\eta n_\epsilon n_\delta n_\eta n_\xi n_\eta,
K = \langle q_1, q_2, q_3 \rangle.
$

We label some roots of $G$ as in Table~\ref{Case18} in Appendix. It can be calculated that $K$ has an orbit $O_{1}=\{1,\cdots, 14\}$ which contains only one non-commuting pair of roots $\{4,9\}$ contributing a correction term that lies in $U_{15}$. Also $\pi(q_1)$ swaps $4$ with $9$, and fixes $15$. Let $v(a)=\prod_{i=1}^{14}\epsilon_i(a)$ and define $H:=v(a)\cdot K$. It is clear that the same arguments work as in the last case.
\section{On a question of K\"{u}lshammer for representations of finite groups in reductive groups}
\subsection{The $E_6$ example}
\begin{proof}[Proof of Theorem~\ref{KulshammerA}]
Let $G$ be a simple simply-connected algebraic group of type $E_6$ defined over $k$.
We keep the notation from Sections~2 and 3. Pick $c\in k$ such that $c^3=1$ and $c\neq 1$. Let
\begin{alignat*}{2}
t_1:&=\alpha^{\vee}(c), \; t_2:=\beta^{\vee}(c), \; t_3:=\gamma^{\vee}(c), \; t_4:=\delta^{\vee}(c), \; t_5:=\epsilon^{\vee}(c),\\
q_1:&=n_\alpha n_\beta n_\gamma n_\beta n_\alpha n_\beta n_\gamma n_\beta n_\gamma n_\delta n_\epsilon n_\delta n_\gamma n_\epsilon,\\
q_2:&=n_\alpha n_\beta n_\gamma n_\delta n_\gamma n_\beta n_\alpha n_\beta n_\delta n_\epsilon n_\delta,\\
H':&=\langle t_1,t_2,t_3,t_4,t_5,q_1,q_2 \rangle.
\end{alignat*}
Note that $q_1$ and $q_2$ here are the same as $q_1$ and $q_2$ in Case $4$ of Section~3. Using Magma, we obtain the defining relations of $H'$:
\begin{alignat*}{2}
&t_i^{3}=1, \; q_1^3=1, \; q_2^2=1, \; q_1\cdot t_1 = (t_1 t_2 t_3)^{-1}, \; q_1\cdot t_2 = t_1 t_2 t_3 t_4 t_5, \; q_1\cdot t_3 = (t_2 t_3 t_4 t_5)^{-1}, \\
&q_1 \cdot t_4 = t_2, \; q_1\cdot t_5 = t_3 t_4, \; q_2\cdot t_1=(t_3 t_4)^{-1},\; q_2\cdot t_2=t_2^{-1}, \; q_2\cdot t_3=t_2 t_3 t_4 t_5, \\
&q_2\cdot t_4= (t_1 t_2 t_3 t_4 t_5)^{-1}, \; q_2\cdot t_5 = t_1 t_2 t_3, \; [t_i, t_j]=1, \; (q_1^2 q_2)^2=1.
\end{alignat*}

Let 
\begin{alignat*}{2}
\Gamma:= F \times C_2 = &\langle r_1, r_2, r_3, r_4, r_5, s_1, s_2, z \mid r_i^3=s_1^3=1, s_2^2=1, s_1 r_1 s_1^{-1}= 
               (r_1 r_2 r_3)^{-1}, \\
                &s_1 r_2 s_1^{-1}= r_1 r_2 r_3 r_4 r_5, s_1 r_3 s_1^{-1}= 
               (r_2 r_3 r_4 r_5)^{-1}, s_1 r_4 s_1^{-1}= r_2, s_1 r_5 s_1^{-1}= r_3 r_4, \\
               &s_2 r_1 s_2^{-1} = (r_3 r_4)^{-1}, s_2 r_2 s_2^{-1}= r_2^{-1}, 
               s_2 r_3 s_2^{-1}= r_2 r_3 r_4 r_5, s_2 r_4 s_2^{-1}=(r_1 r_2 r_3 r_4 r_5)^{-1},\\
               & s_2 r_5 s_2^{-1}= r_1 r_2 r_3, [r_i, r_j] = (s_1^2 s_2)^2 = [r_i, z] = [s_i, z] = 1 \rangle.
\end{alignat*}
Then $F\cong 3^{1+2}:3^2:S_3$ and $|F|=1458=2\times 3^6$. Let $\Gamma_2:=\langle s_2, z \rangle$ (a Sylow $2$-subgroup of $\Gamma$). It is clear that $F\cong H'$. 

For any $a\in k$ define $\rho_a \in \textup{Hom}(\Gamma, G)$ by
\begin{equation*}
\rho_a(r_i)=t_i,\; \rho_a(s_1)=q_1,\; \rho_a(s_2)=q_2 \epsilon_{21}(a),\; \rho_a(z)=\epsilon_{21}(1).
\end{equation*}
It is easily checked that this is well-defined.
\begin{lem}\label{Gconj}
$\rho_a|_{\Gamma_2}$ is $G$-conjugate to $\rho_b|_{\Gamma_2}$ for any $a,b\in k$.
\end{lem}
\begin{proof}
It is enough to prove that $\rho_0|_{\Gamma_2}$ is $G$-conjugate to $\rho_a|_{\Gamma_2}$ for any $a\in k$. Now let
\begin{equation*}
u(\sqrt a)=\epsilon_{7}(\sqrt a)\epsilon_{8}(\sqrt a).
\end{equation*}
Then an easy computation shows that 
\begin{equation*}
u(\sqrt a)\cdot q_2 = q_2\epsilon_{21}(a),\; u(\sqrt a)\cdot \epsilon_{21}(1)=\epsilon_{21}(1). 
\end{equation*}
So we have 
\begin{equation*}
u(\sqrt a)\cdot (\rho_0|_{\Gamma_2})=\rho_{a}|_{\Gamma_2}.
\end{equation*} 
\end{proof}

\begin{lem}\label{nonGconj}
$\rho_a$ is not $G$-conjugate to $\rho_b$ for $a\neq b$.
\end{lem}
\begin{proof}
Let $a,b\in k$. Suppose that there exists $g\in G$ such that $g\cdot \rho_a =\rho_b$. Since $\rho_a(r_i)=t_i$, we need $g\in C_G(t_1,t_2,t_3,t_4,t_5)$. A direct computation shows that $C_G(t_1,t_2,t_3,t_4,t_5)=TG_{21}$. So let $g=t m$ for some $t \in T$ and $m\in G_{21}$. Note that $q_2$ centralizes $G_{21}$. So, 
\begin{alignat}{2}
(t q_2 t^{-1})(t m \epsilon_{21}(a) m^{-1} t^{-1})  &= (t m) q_2 \epsilon_{21}(a) (m^{-1} t^{-1})\nonumber\\
                        &= g\cdot \rho_a(s_2)\nonumber\\
                        &= \rho_b(s_2)\nonumber \\
                        & = q_2\epsilon_{21}(b). \label{equation1}
\end{alignat}
Note that $t q_2 t^{-1}\in G_{\alpha\beta\gamma\delta\epsilon}$ and $t m \epsilon_{21}(a)m^{-1}t^{-1}\in G_{21}$. Since $[G_{\alpha\beta\gamma\delta\epsilon}, G_{21}]=1$, it is clear that $G_{\alpha\beta\gamma\delta\epsilon}\cap G_{21}=1$. Now (\ref{equation1}) yields that $t q_2 t^{-1} = q_2$. We also have 
\begin{equation*}
 q_1 = \rho_b(s_1) = g\cdot \rho_a(s_1) = t m \cdot q_1 = t q_1 t^{-1}.
\end{equation*}
So $t$ commutes with $q_1$ and $q_2$. Then a quick calculation shows that $t\in G_{21}$. So $g\in G_{21}$. But $G_{21}$ is a simple group of type $A_1$, so the pair $(q_2 \epsilon_{21}(a), \epsilon_{21}(1))$ is not $G_{21}$-conjugate to $(q_2 \epsilon_{21}(b), \epsilon_{21}(1))$ if $a\neq b$. Therefore $\rho_a$ is not $G$-conjugate to $\rho_b$ if $a\neq b$.
\end{proof}
Now Theorem~\ref{KulshammerA} follows from Lemmas~\ref{Gconj} and \ref{nonGconj}.
\end{proof}


\begin{rem}
One can obtain examples with the same properties as in Theorem~\ref{KulshammerA} for $G=E_7, E_8$ using the $E_7$ and $E_8$ examples in Sections~4 and 5.
\end{rem}


\subsection{The non-connected $A_2$ example}

\begin{proof}[Proof of Theorem~\ref{KulshammerB}]
We have $G^{\circ}=SL_3(k)$. Fix a maximal torus $T$ of $G^{\circ}$, and a Borel subgroup of $G^{\circ}$ containing $T$. Let $\{\alpha, \beta\}$ be the set of simple roots of $G^{\circ}$. Let $c\in k$ such that $|c|=d$ is odd and $c\neq 1$. Define $t:=(\alpha-\beta)^{\vee}(c)$. For each $a\in k$, define $\rho_a\in \textup{Hom}(\Gamma, G)$ by 
\begin{equation*}
\rho_a(r)=t, \; \rho_a(s)=\sigma \epsilon_{\alpha+\beta}(a), \; \rho_a(z)=\epsilon_{\alpha+\beta}(1).
\end{equation*}
An easy computation shows that this is well-defined.
\begin{lem}\label{Kulnonconnected1}
$\rho_a\mid_{\Gamma_2}$ is $G$-conjugate to $\rho_b\mid_{\Gamma_2}$ for any $a, b\in k$.
\end{lem} 
\begin{proof}
Let $u(\sqrt a):=\epsilon_\alpha(\sqrt a)\epsilon_\beta(\sqrt a)$. Then
\begin{equation*}
u(\sqrt a)\cdot \sigma = \sigma \epsilon_{\alpha+\beta}(a),\; u(\sqrt a)\cdot \epsilon_{\alpha+\beta}(1)=\epsilon_{\alpha+\beta}(1).
\end{equation*} 
This shows that $u(\sqrt a)\cdot (\rho_0\mid_{\Gamma_2}) = \rho_a\mid_{\Gamma_2}$. 
\end{proof}
\begin{lem}\label{Kulnonconnected2}
$\rho_a$ is not $G$-conjugate to $\rho_b$ if $a\neq b$. 
\end{lem}
\begin{proof}
Let $a, b \in k$. Suppose that there exists $g\in G$ such that $g\cdot \rho_a = \rho_b$. Since $\rho_a(r)=t$, we have $g\in C_G(t)=T G_{\alpha+\beta}$. So let $g=hm$ for some $h\in T$ and $m\in G_{\alpha+\beta}$. We compute
\begin{alignat}{2}\label{A2comp}
(h\sigma h^{-1})(hm \epsilon_{\alpha+\beta}(a)m^{-1}h^{-1}) &=(hm)\sigma \epsilon_{\alpha+\beta}(a) (m^{-1}h^{-1}) \nonumber\\
                       &=g\cdot \rho_a(s) \nonumber\\
                       &=\rho_b(s) \nonumber\\
                       &=\sigma \epsilon_{\alpha+\beta}(b).
\end{alignat}
Now (\ref{A2comp}) shows that $h$ commutes with $\sigma$. Then $h$ is of the form $h:=(\alpha+\beta)^{\vee}(x)$ for some $x\in k^*$. So $h\in G_{\alpha+\beta}$. Thus $g\in G_{\alpha+\beta}$. 
But $G_{\alpha+\beta}$ is a simple group of type $A_1$, so the pair $(\sigma\epsilon_{\alpha+\beta}(a), \epsilon_{\alpha+\beta}(1))$ is not $G_{\alpha+\beta}$-conjugate to $(\sigma\epsilon_{\alpha+\beta}(b), \epsilon_{\alpha+\beta}(1))$ unless $a=b$. So $\rho_a$ is not $G$-conjugate to $\rho_b$ unless $a=b$.
\end{proof}
Theorem~\ref{KulshammerB} follows from Lemmas~\ref{Kulnonconnected1} and \ref{Kulnonconnected2}.
\end{proof}

\section*{Acknowledgements}
This research was supported by Marsden Grant UOA1021. The author would like to thank Ben Martin and Don Taylor for helpful discussions. 
\section*{Appendix}
\begin{table}[h]
\centering
\scalebox{0.6}{
\begin{tabular}{ccccc}
\rootsE{1}{1}{0}{0}{0}{0}{0}&\rootsE{2}{1}{1}{1}{1}{1}{0}&\rootsE{3}{1}{0}{1}{1}{1}{1}&\rootsE{4}{1}{1}{1}{2}{1}{0}&\rootsE{5}{1}{0}{1}{2}{1}{1}\\
&&&\\
\rootsE{6}{1}{1}{2}{3}{2}{1}&\rootsE{7}{1}{0}{0}{1}{0}{0}&\rootsE{8}{1}{1}{2}{2}{2}{1}&\rootsE{9}{1}{0}{1}{1}{0}{0}&\rootsE{10}{1}{0}{0}{1}{1}{0}\\
&&&\\
\rootsE{11}{1}{0}{1}{1}{1}{0}&\rootsE{12}{1}{1}{1}{2}{1}{1}&\rootsE{13}{1}{1}{2}{2}{1}{1}&\rootsE{14}{1}{1}{1}{2}{2}{1}&\rootsE{15}{1}{1}{1}{1}{0}{0}\\
&&&\\
\rootsE{16}{1}{0}{0}{1}{1}{1}&\rootsE{17}{1}{0}{1}{2}{1}{0}&\rootsE{18}{1}{1}{1}{1}{1}{1}&\rootsE{19}{1}{1}{2}{2}{1}{0}&\rootsE{20}{1}{0}{1}{2}{2}{1}\\
&&&\\
\rootsE{21}{2}{1}{2}{3}{2}{1}&\rootsE{22}{0}{1}{0}{0}{0}{0}&\rootsE{23}{0}{0}{1}{0}{0}{0}&\rootsE{24}{0}{0}{0}{1}{0}{0}&\rootsE{25}{0}{0}{0}{0}{1}{0}\\
&&&\\
\rootsE{26}{0}{0}{0}{0}{0}{1}&\rootsE{27}{0}{1}{1}{0}{0}{0}&\rootsE{28}{0}{0}{1}{1}{0}{0}&\rootsE{29}{0}{0}{0}{1}{1}{0}&\rootsE{30}{0}{0}{0}{0}{1}{1}\\
&&&\\
\rootsE{31}{0}{1}{1}{1}{0}{0}&\rootsE{32}{0}{0}{1}{1}{1}{0}&\rootsE{33}{0}{0}{0}{1}{1}{1}&\rootsE{34}{0}{1}{1}{1}{1}{0}&\rootsE{35}{0}{0}{1}{1}{1}{1}\\
&&&\\
\rootsE{36}{0}{1}{1}{1}{1}{1}
\end{tabular}
}
\caption{The set of positive roots of $E_6$}\label{E6}
\end{table}
\begin{table}[h]
\centering
\scalebox{0.6}{
\begin{tabular}{llll}
\rootsG{1}{1}{1}{1}{1}{1}{0}{0}&\rootsG{2}{1}{0}{1}{1}{2}{1}{1}&\rootsG{3}{1}{1}{2}{2}{2}{1}{1}&\rootsG{4}{1}{0}{1}{1}{1}{1}{0}\\
&&&\\
\rootsG{5}{1}{1}{2}{2}{2}{2}{1}&\rootsG{6}{1}{1}{1}{2}{3}{2}{1}&\rootsG{7}{1}{0}{0}{1}{1}{0}{0}&\rootsG{8}{1}{0}{0}{0}{1}{1}{0}\\
&&&\\
\rootsG{9}{1}{0}{0}{0}{1}{1}{1}&\rootsG{10}{1}{1}{1}{2}{2}{2}{1}&\rootsG{11}{1}{0}{1}{2}{3}{2}{1}&\rootsG{12}{1}{0}{0}{1}{1}{1}{1}\\
&&&\\
\rootsG{13}{1}{1}{1}{1}{2}{1}{0}&\rootsG{14}{1}{0}{1}{2}{2}{1}{0}&\rootsG{15}{2}{1}{2}{3}{4}{3}{2}
\end{tabular}
}
\caption{Case $2$ ($E_7$)}
\label{Case14}
\end{table}
\begin{table}[h!]
\centering
\scalebox{0.6}{
\begin{tabular}{cccc}
\rootsC{1}{1}{0}{0}{0}{0}{0}{0}{0}&\rootsC{2}{1}{0}{1}{1}{1}{1}{0}{0}&\rootsC{3}{1}{0}{0}{0}{0}{1}{1}{1}&\rootsC{4}{1}{0}{1}{1}{2}{2}{1}{0}\\
&&&\\
\rootsC{5}{1}{1}{1}{1}{1}{2}{2}{1}&\rootsC{6}{1}{1}{1}{1}{2}{3}{2}{1}&\rootsC{7}{1}{1}{2}{2}{3}{3}{2}{1}&\rootsC{8}{2}{0}{1}{1}{2}{3}{2}{1}
\end{tabular}
}
\caption{Case $1$ ($E_8$)}
\label{Case9}
\end{table}
\begin{table}[h!]
\centering
\scalebox{0.6}{
\begin{tabular}{cccc}
\rootsC{1}{1}{0}{0}{0}{0}{1}{0}{0}&\rootsC{2}{1}{0}{0}{0}{1}{1}{0}{0}&\rootsC{3}{1}{0}{0}{0}{0}{1}{1}{0}&\rootsC{4}{1}{0}{1}{1}{1}{1}{1}{0}\\
&&&\\
\rootsC{5}{1}{0}{1}{1}{1}{1}{1}{1}&\rootsC{6}{1}{0}{0}{0}{1}{2}{2}{1}&\rootsC{7}{1}{1}{1}{1}{1}{1}{1}{1}&\rootsC{8}{1}{1}{1}{1}{2}{2}{1}{0}\\
&&&\\
\rootsC{9}{1}{1}{1}{1}{1}{2}{1}{1}&\rootsC{10}{1}{0}{1}{1}{2}{2}{1}{1}&\rootsC{11}{1}{1}{2}{2}{2}{2}{1}{0}&\rootsC{12}{1}{1}{1}{1}{2}{2}{2}{1}\\
&&&\\
\rootsC{13}{1}{0}{1}{1}{2}{3}{2}{1}&\rootsC{14}{1}{1}{2}{2}{2}{3}{2}{1}&\rootsC{15}{2}{1}{2}{2}{2}{3}{2}{1}
\end{tabular}
}
\caption{Case $2$ ($E_8$)}
\label{Case18}
\end{table}
\newpage
\bibliography{mybib}

\end{document}